\documentclass[11pt]{amsart}

\newcommand{\Addresses}{{
  \bigskip
  \footnotesize  
  
  \noindent Gabriele Viaggi, \textsc{Mathematical Institute, Sapienza University of Rome, Rome}\par\nopagebreak
  \textit{E-mail address}: \texttt{gabriele.viaggi@uniroma1.it}
  }}

\usepackage[margin=1.5in]{geometry}
\usepackage[colorlinks,citecolor=cyan,linkcolor=blue]{hyperref}

\usepackage{dsfont}
\usepackage{amsmath}
\usepackage{amsthm}
\usepackage{amssymb}
\usepackage[all,cmtip]{xy}
\usepackage{enumerate}
\usepackage{enumitem}
\usepackage{overpic}
\usepackage{float}

\newcommand{\mb}[1]{
\mathbb{#1}
}
\newcommand{\T}{
\mathcal{T}
}

\newtheorem*{thm*}{Theorem}
\newtheorem*{cor*}{Corollary}
\newtheorem*{pro*}{Proposition}

\newtheorem{thmA}{Theorem}

\newtheorem{thm}{Theorem}[section]

\newtheorem{cor}[thm]{Corollary}

\newtheorem{lem}[thm]{Lemma}
\newtheorem{pro}[thm]{Proposition}

\theoremstyle{definition}
\newtheorem{claim}[thm]{Claim}

\newtheorem*{defi*}{Definition}
\newtheorem{dfn}[thm]{Definition}

\newtheorem{remark}[thm]{Remark}

\title[Geometric components via robust families]{Geometric components of representation spaces via robust families of submanifolds}
\author{Gabriele Viaggi}

\begin{document}

\begin{abstract}
We introduce robust families of submanifolds for a linear Lie group $G$. We show that they give rise to geometric subspaces of the representation space ${\rm Hom}(\Gamma,G)$. As an application, we give a unified short proof of results of Beyrer and Kassel and of Benoist and Koszul about the existence of higher Teichmüller components for $G={\rm SO}(p,q+1),{\rm SL}(p+1,\mb{R})$. Being based on very general principles, our approach might be suited for finding geometric components in various ${\rm Hom}(\Gamma,G)$.
\end{abstract}

\maketitle

\section{Introduction}

In this article, we develop a novel geometric approach to geometric components of representation spaces ${\rm Hom}(\Gamma,G)$ (where $\Gamma$ is a finitely generated group and $G$ is a Lie group). These are connected components of ${\rm Hom}(\Gamma,G)$ consisting entirely of discrete and faithful representations. When $\Gamma$ is the fundamental group of a closed orientable surface of genus at least 2, the search and classification of such components generated a rich and active area of research going under the name of higher Teichmüller theory with plenty of interesting examples (see \cite{BIW,La,FG,BGLPW}). Much less is known when $\Gamma$ is the fundamental group of a closed aspherical manifold of dimension at least 3 (see \cite{Be3,K,Ba,BK}). A crucial challenge comes from the fact that most techniques used in the surface case rely specifically on either surface topology or the non-abelian Hodge correspondence, and both break down in higher dimensions. Developing a framework to produce and test examples in such generalities is one of the central motivations for this work. The main application is a new unified short proof of the following results for $G={\rm SO}(p,q+1)$ and $G={\rm SL}(p+1,\mb{R})$ originally due to Beyrer and Kassel \cite{BK} (the first) and Benoist \cite{Be3} and Koszul \cite{K} (the second). 

\begin{thmA}
\label{thm:main2}
Fix $p\ge 2,q\ge 1$. Let $\Gamma$ be a torsion-free group of cohomological dimension $p$ without infinite nilpotent normal subgroups. The subspace 
\[
{\rm Max}(\Gamma,{\rm SO}(p,q+1))\subset{\rm Hom}(\Gamma,{\rm SO}(p,q+1))
\]
of discrete and faithful representations that preserve some complete {\rm maximal $p$-submanifold} $M\subset\mb{H}^{p,q}$ is open and closed.   
\end{thmA}

\begin{thmA}
\label{thm:main3}
Fix $p\ge 2$. Let $\Gamma$ be a torsion-free group of cohomological dimension $p$ without infinite nilpotent normal subgroups. The subspace 
\[
{\rm AffSph}(\Gamma,{\rm SL}(p+1,\mb{R}))\subset{\rm Hom}(\Gamma,{\rm SL}(p+1,\mb{R}))
\]
of discrete and faithful representations that preserve some complete {\rm hyperbolic affine $p$-sphere} $M\subset\mb{R}^{p+1}$ is open and closed. 
\end{thmA}

Key to our proof is the notion of a robust family, which we now describe.

\begin{dfn}[Robust Family]
\label{dfn:robust family}    
Let $G<{\rm SL}(d,\mb{R})$ be a linear Lie group. A {\em robust family of $p$-submanifolds} for $G$ is a collection $\mathcal{M}$ of smooth contractible $p$-submanifolds $M\subset\mb{R}^d$ endowed with complete Riemannian metrics $\sigma_M$ that satisfies the following conditions.
\begin{enumerate}
    \item\label{lab:invariance}{{\slshape Invariance}. If $M\in\mathcal{M}$ and $g\in G$ then $g(M)\in\mathcal{M}$ and $\sigma_{g(M)}=(g^{-1})^*\sigma_M$.}
    \item\label{lab:compactness}{{\slshape Compactness}. Consider
    \[
    \mathcal{PM}:=\{(M,o)\,|\,M\in\mathcal{M}, o\in M\}.
    \]
    For every sequence $(M_n,o_n)\in\mathcal{PM}$ there is a sequence $g_n\in G$ and $(M,o)\in\mathcal{PM}$ such that $(g_n(M_n),g_n(o_n))$ converges to $(M,o)$ in the pointed $\mathcal{C}^2$-topology on $\mathcal{PM}$ (see Definition \ref{dfn:pointed topology} below).}  
    \item\label{lab:avoidance}{{\slshape Avoidance}. Let $\partial G\subset\mb{P}{\rm M}(d,\mb{R})$ denote the boundary of the projection of $G$ to the projectivization of the space of $d\times d$ matrices ${\rm M}(d,\mb{R})$. Consider the family of subspaces
    \[
    \mathcal{K}:=\{{\rm Ker}(\phi)\subset\mb{R}^d\,|\,\phi\in\partial G\subset\mb{P}{\rm M}(d,\mb{R})\}.
    \]
    We have $M\subsetneq V$ for every $V\in\mathcal{K}$ and $M\in\mathcal{M}$.}
    \item\label{lab:domination}{{\slshape Domination}. There exists a continuous function
    \[
    d_\bullet:\mathcal{PM}\to\mathcal{C}(\mb{R}^d,[0,\infty))
    \]
    with the following properties. 
    \begin{enumerate}
        \item\label{lab:domination1}{$d_{g(M,o)}\circ g=d_{(M,o)}$ for every $g\in G, (M,o)\in\mathcal{PM}$.}
        \item\label{lab:domination2}{$d_M(o,x)\le d_{(M,o)}(x)$ for every $x\in M$.}      
        \item\label{lab:domination3}{Let $\Gamma$ be a finitely generated group. For every convergent sequence of representations $\rho_n:\Gamma\to G$ such that $\rho_n(\Gamma)$ preserves $M_n\in\mathcal{M}$ there exists a choice of basepoints $o_n\in M_n$ such that for every $\gamma\in\Gamma$
        \[
        d_{(M_n,o_n)}(\rho_n(\gamma)o_n)
        \]
        is uniformly bounded.}
    \end{enumerate}
    }
\end{enumerate}
\end{dfn}

The pointed $\mathcal{C}^2$-topology we put on the space $\mathcal{PM}$ in Property \eqref{lab:compactness} is slightly different from the one we would get by considering each $(M,o)\in\mathcal{PM}$ as an abstract complete Riemannian manifold (see for example \cite[Chapter 10]{Pe}) since we want to keep track also of the embedding $M\to\mb{R}^d$.

\begin{dfn}[Pointed $\mathcal{C}^2$-Topology]
\label{dfn:pointed topology}
A sequence $(M_n,o_n)\in\mathcal{PM}$ {\em converges to $(M,o)\in\mathcal{PM}$ in the pointed $\mathcal{C}^2$-topology} if for every $R>0$ and $n$ large enough there exists an open subset $\Omega\subset M$ containing $B(o,R)$ and a smooth 2-bilipschitz embedding $f_n:\Omega\to M_n$ such that 
\begin{enumerate}
    \item{$f_n(o)=o_n$.}
    \item{$B(o_n,R)\subset f_n(\Omega)$.}
    \item{$||f_n^*\sigma_{M_n}-\sigma_M||_{\mathcal{C}^2(B(o,R))}\to 0$.}
    \item{$f_n\to\iota$ on $B(o,R)$ where $\iota:M\to\mb{R}^d$ is the inclusion.}
\end{enumerate}
\end{dfn}

The following result shows that if a linear Lie group $G<{\rm SL}(d,\mb{R})$ admits a robust family of $p$-submanifolds then it admits geometric-like subspaces in ${\rm Hom}(\Gamma,G)$.

\begin{thmA}
\label{thm:main4}
Let $G<{\rm SL}(d,\mb{R})$ be a linear Lie group and $\mathcal{M}$ a {\rm robust family of $p$-submanifolds}. Let $\Gamma$ be a finitely generated group. Consider
\[
\mathcal{T}(\Gamma,G,\mathcal{M}):=\left\{\rho\in{\rm Hom}(\Gamma,G)\,\left|\,
\begin{array}{c}
     \text{\rm $\rho(\Gamma)$ preserves some $M\in\mathcal{M}$}
\end{array}\right.\right\}.
\]

If $\rho_n\in\T(\Gamma,G,\mathcal{M})$ is a sequence that converges in ${\rm Hom}(\Gamma,G)$, then there exists a sequence $g_n\in G$ such that, up to passing to subsequences, $\bar{\rho}_n:=g_n\rho_ng_n^{-1}$ converges to a representation $\bar{\rho}\in\T(\Gamma,G,\mathcal{M})$. In particular, the image of the projection of $\mathcal{T}(\Gamma,G,\mathcal{M})$ to the character variety ${\rm Hom}(\Gamma,G)\,//\,G$ is closed.
\end{thmA}

Note that there is no restriction on the group $\Gamma$ and on the representations besides the fact that they preserve some $M\in\mathcal{M}$ (for example, they might fail to be injective or have discrete image). In the case where  $\Gamma$ does not have any infinite nilpotent normal subgroup and each $\rho_n$ is discrete and faithful, the limit $\bar{\rho}$ is also discrete and faithful by standard consequences of the Kazhdan–Margulis–Zassenhaus Theorem. In particular, under the same assumptions on $\Gamma$, the image of the projection of the discrete and faithful locus of $\T(\Gamma,G,\mathcal{M})$ to the character variety is also closed. Theorem \ref{thm:main4} also extends some work of Cooper and Tillmann \cite{CT}. Let us mention that they also obtain, as an application, a short proof of Benoist and Koszul's result.

\subsection*{Comments}
The definition of a robust family is inspired by the properties shared by many spaces of solutions of suitable {\em asymptotic Plateau problems}. Important examples are the ones giving rise to maximal $p$-submanifolds in $\mb{H}^{p,q}$ (as studied by Labourie, Toulisse, and Wolf \cite{LTW} and Seppi, Smith, and Toulisse \cite{SST}) and hyperbolic affine $p$-spheres in $\mb{R}^{p+1}$ (as studied by Cheng and Yau \cite{CY1,CY2}). Our two main applications are for these two classes. In those cases, the dominating functions of Property \eqref{lab:domination} will be essentially logarithms of linear functions. This highlights a crucial point: While the relation between representations $\rho\in\T(\Gamma,G,\mathcal{M})$ and Riemannian metrics $(M,\sigma_M)$ can be quite involved, their interaction with linear functionals is straightforward.

In order to pass from Theorem \ref{thm:main4} to closed geometric subspaces of ${\rm Hom}(\Gamma,G)$, one has to argue that if each $\rho_n$ is discrete and faithful then, under suitable assumptions on $\Gamma$ (such as having cohomological dimension $p$ and no infinite nilpotent normal subgroup), the sequence of renormalizing elements $g_n$ must be bounded. In the settings of Theorems \ref{thm:main2} and \ref{thm:main3}, this is the only place where we use the fact that the action $\bar{\rho}(\Gamma)\curvearrowright M$ is properly discontinuous and cocompact. 

Lastly, let us comment openness in Theorems \ref{thm:main2} and \ref{thm:main3}. This is very general and boils down to a combination of the Ehresmann-Thurston principle and the existence of solutions of the releveant asymptotic Plateau problems \cite{SST,CY1,CY2}. We emphasize that the existence of solutions does not enter the proof of closedness which is the core of the argument.

Being based on very general principles, we hope that our approach might also be applicable to find other geometric components in ${\rm Hom}(\Gamma,G)$ for different Lie groups $G$ by considering suitable robust families in homogeneous spaces $G/H$.

\subsection*{Acknowledgements}
I would like to thank Andrea Seppi, Jeremy Toulisse, Jonas Beyrer, Fanny Kassel, Nicolas Tholozan, and Beatrice Pozzetti for useful discussions and generous feedback on a first draft of this paper.

\section{The proof of Theorem \ref{thm:main4}}
\label{sec:proof main1}

We prove the following more precise version of Theorem \ref{thm:main4}.

\begin{thm}
\label{thm:main4precise}
Let $G<{\rm SL}(d,\mb{R})$ be a linear Lie group and $\mathcal{M}$ a {\rm robust family of $p$-submanifolds} for $G$ (see Definition \ref{dfn:robust family}). Let $\Gamma$ be a finitely generated group. Consider a converging sequence $\rho_n:\Gamma\to G$ such that each $\rho_n(\Gamma)$ preserves some $M_n\in\mathcal{M}$. Let $o_n\in M_n$ be the basepoints provided by Property \eqref{lab:domination3}. Let $g_n\in G$ be elements such that $(g_n(M_n),g_n(o_n))\to(M,o)$ in $\mathcal{PM}$ as given by Property \eqref{lab:compactness}. Set $\bar{\rho}_n:=g_n\rho_ng_n^{-1}$. Then, up to passing to subsequences, $\bar{\rho}_n$ converges to a representation $\bar{\rho}$ preserving $M\in\mathcal{M}$. 
\end{thm}

\begin{proof}[Proof of Theorem \ref{thm:main4precise}]
Note that by Property \eqref{lab:invariance}, $\rho_n(\Gamma)$ acts by isometries on $M_n$.

By assumption, $(g_n(M_n),g_n(o_n))$ converges to $(M,o)$ in the pointed $\mathcal{C}^2$-topology on $\mathcal{PM}$. We recall that this means that for every $R>0$ and $n$ large enough there exists an open subset $\Omega\subset M$ containing $B(o,R)$ and a smooth 2-bilipschitz embedding $f_n:\Omega\to g_n(M_n)$ such that 
\begin{enumerate}[label={(\rm C\arabic*)},ref=C\arabic*]
    \item\label{lab:c1}{$f_n(o)=g_n(o_n)$.}
    \item\label{lab:c2}{$B(g_n(o_n),R)\subset f_n(\Omega)$.}
    \item\label{lab:c3}{$||f_n^*\sigma_{g_n(M_n)}-\sigma_M||_{\mathcal{C}^2(B(o,R))}\to 0$.}
    \item\label{lab:c4}{$f_n\to\iota$ on $B(o,R)$ where $\iota:M\to\mb{R}^d$ is the inclusion.}
\end{enumerate}

Again, Property \eqref{lab:invariance} implies that $\bar{\rho}_n(\Gamma)$ acts by isometries on $g_n(M_n)$. Our goal is to show that, up to passing to subsequences, the representations $\bar{\rho}_n$ converge to $\bar{\rho}_n\to\bar{\rho}$ preserving $M$. 

In order to get a limit representation $\bar{\rho}$, it is enough to show that $\bar{\rho}_n(\gamma)$ is bounded for every $\gamma\in\Gamma$. A standard diagonal argument would then provide a convergent subsequence $\bar{\rho}_n\to\bar{\rho}$. To show that $\bar{\rho}_n(\gamma)$ is bounded we will exploit the pointed $\mathcal{C}^2$-convergence of $(g_n(M_n),g_n(o_n))\to(M,o)$ and the domination properties. First, we need some estimates. For every $\gamma\in\Gamma$ we have 
\begin{align*}
d_{g_n(M_n)}(g_n(o_n),\bar{\rho}_n(\gamma)g_n(o_n)) &=d_{M_n}(o_n,\rho_n(\gamma)o_n) &\text{\rm by Property \eqref{lab:invariance}}\\
 &\le d_{(M_n,o_n)}(\rho_n(\gamma)o_n) &\text{\rm by Property \eqref{lab:domination2}}.
\end{align*}
Thus, by Property \eqref{lab:domination3}, $d_{g_n(M_n)}(g_n(o_n),\bar{\rho}_n(\gamma)g_n(o_n))$ is uniformly bounded, say by $B>0$. As a consequence, by the triangle inequality, for every $y\in g_n(M_n)$ we have
\begin{align*}
d_{g_n(M_n)}(g_n(o_n),\bar{\rho}_n(\gamma)y) &\le d_{g_n(M_n)}(g_n(o_n),\bar{\rho}_n(\gamma)g_n(o_n))+d_{g_n(M_n)}(\bar{\rho}_n(\gamma)g_n(o_n),\bar{\rho}_n(\gamma)y)\\
&\le B+d_{g_n(M_n)}(g_n(o_n),y).    
\end{align*}

Suppose by contradiction that $\bar{\rho}_n(\gamma)$ is unbounded. Then $a_n\bar{\rho}_n\to\phi_\gamma\in\partial G$ for some $a_n\to 0$. We will deduce that $M\subset{\rm Ker}(\phi_\gamma)$ contradicting Property \eqref{lab:avoidance}. Let $r>0$ be an arbitrary radius. Choose $R$ much larger than $2r+2B$. For $n$ large enough, let $f_n:\Omega\to M_n$ be the 2-bilipschitz approximating map defined over $B(o,R)\subset\Omega$ provided by the pointed $\mathcal{C}^2$-convergence. Note that, by \eqref{lab:c1} and the fact that $f_n$ is 2-bilipschitz, for every $x\in B(o,r)$ we have 
\[
d_{g_n(M_n)}(f_n(x),g_n(o_n))=d_{g_n(M_n)}(f_n(x),f_n(o))\le 2d_M(x,o)
\]
and, hence, $d_{g_n(M_n)}(g_n(o_n),\bar{\rho}_n(\gamma)f_n(x))\le B+2d_M(x,o)<R$ so that, by \eqref{lab:c2}, $f_n^{-1}\bar{\rho}_n(\gamma)f_n$ is well-defined on $B(o,r)$. We have 
\begin{align*}
&d_M(f_n^{-1}\bar{\rho}_n(\gamma)f_n(x),o) &\\
&\le 2d_{g_n(M_n)}(\bar{\rho}_n(\gamma)f_n(x),f_n(o)) &\text{\rm $f_n$ 2-bilipschitz}\\
&\le 2d_{g_n(M_n)}(\bar{\rho}_n(\gamma)o_n,f_n(o))+2d_{g_n(M_n)}(\bar{\rho}_n(\gamma)f_n(x),\bar{\rho}_n(\gamma)o_n) &\text{\rm triangle inequality}\\
&=2d_{g_n(M_n)}(\bar{\rho}_n(\gamma)f_n(o),f_n(o))+2d_{g_n(M_n)}(f_n(x),f_n(o)) &\text{\rm $\bar{\rho}_n(\gamma)$ isometry}\\
&\le2d_{g_n(M_n)}(\bar{\rho}_n(\gamma)o_n,o_n)+2d_{M}(x,o) &\text{\rm \eqref{lab:c1}+$f_n$ 2-bilipschitz}\\
&\le 2B+2d_M(x,o)<2B+2r.
\end{align*}

We conclude that $f_n^{-1}\bar{\rho}_n(\gamma)f_n(x)$ is uniformly bounded and, hence, it converges up to subsequences to some $y\in M$ (as $(M,\sigma_M)$ is complete). 

Consider now $x\in B(o,r)$ and $a_n\bar{\rho}_n(\gamma)f_n(x)=a_n f_n(f_n^{-1}\bar{\rho}_n(\gamma)f_n)(x)$. By \eqref{lab:c4}, we have $f_n\to\iota$, hence the left-hand side converges to $\phi_\gamma(x)$. Instead, the right-hand side converges to 0 as $a_n\to 0,f_n\to\iota,f_n^{-1}\bar{\rho}_n(\gamma)f_n(x)\to y$. Therefore $x\in{\rm Ker}(\phi_\gamma)$ for every $x\in B(o,r)$. As $r$ was arbitrary, it follows that $M\subset{\rm Ker}(\phi_\gamma)$ giving the desired contradiction. 

Thus, for every $\gamma\in\Gamma$ the sequence $\bar{\rho}_n(\gamma)$ must be bounded. Additionally, the above argument shows also that for every $x\in M$ and $n$ large enough $f_n^{-1}\bar{\rho}_n(\gamma)f_n(x)$ is well-defined. Up to passing to subsequences, we can assume that $\bar{\rho}_n(\gamma)\to\bar{\rho}(\gamma)$ for every $\gamma\in\Gamma$. As $f_n$ converges to the inclusion $\iota:M\to\mb{R}^d$ and $f_n^{-1}\bar{\rho}_n(\gamma)f_n(x)\in M$, we conclude that $\bar{\rho}(\gamma)x\in M$ for every $x\in M$, or, in other words, that $M$ is $\bar{\rho}(\Gamma)$-invariant. This finishes the proof of the theorem.
\end{proof}

\begin{remark}
\label{rmk:discrete and faithful}
By classical consequences of the Kazhdan–Margulis–Zassenhaus Theorem, if $\Gamma$ does not have any infinite nilpotent normal subgroup and each $\rho_n$ is discrete and faithful, then $\bar{\rho}$ is also discrete and faithful.    
\end{remark}

\subsection{The case of discrete representations}
While the previous discussion did not require discreteness of the representations, in our main applications we will have to keep track of this property. It is worth observing that for robust families, discreteness in the ambient Lie group is essentially equivalent to Riemannian discreteness as the following proposition shows.

\begin{pro}
\label{pro:discrete}
Let $G<{\rm SL}(d,\mb{R})$ be a linear Lie group and $\mathcal{M}$ a {\rm robust family of $p$-submanifolds} for $G$ (see Definition \ref{dfn:robust family}). Fix $(M,\sigma_M)\in\mathcal{M}$ and consider ${\rm Stab}(M):=\{g\in G\,|\,g(M)=M\}$. By Property \eqref{lab:invariance}, there is a continuous homomorphism $\mu:{\rm Stab}(M)\to{\rm Isom}(M,\sigma_M)$. Then $\mu$ is proper.
\end{pro}

\begin{proof}
Being a continuous homomorphism between locally compact topological groups, $\mu$ is proper if and only if ${\rm Ker}(\mu)$ is compact. Suppose by contradiction that we have an unbounded sequence $\gamma_n\in{\rm Ker}(\mu)$. Then there is a sequence $a_n\to 0$ such that $a_n\gamma_n\to\phi$ that projects to $\partial G$ in $\mb{P}{\rm M}(d,\mb{R})$. By Property \eqref{lab:avoidance}, there is a point $x\in M$ not contained in ${\rm Ker}(\phi)$. On the one hand, we have $a_n\gamma_n x\to\phi(x)$. On the other hand, $a_n\gamma_nx=a_nx\to 0$ since $a_n\to 0$. Thus $x\in{\rm Ker}(\phi)$ and we reached a contradiction. 
\end{proof}

\begin{cor}
\label{cor:discrete and faithful}
Let $\Gamma$ be a torsion-free group. Then $\rho:\Gamma\to{\rm Stab}(M)$ is discrete and faithful if and only if $\mu\rho:\Gamma\to{\rm Isom}(M,\sigma_M)$ is discrete and faithful.    
\end{cor}

Recalling that manifolds $M$ in robust families $\mathcal{M}$ are contractible and that discrete and torsion-free subgroups of ${\rm Isom}(M,\sigma_M)$ act properly discontinuously and freely on $M$, we obtain the following.

\begin{cor}
\label{cor:cohomological dimension}
Let $\Gamma$ be a torsion-free group of cohomological dimension $p$. If $\rho:\Gamma\to G$ is a discrete and faithful representation preserving a $p$-submanifold $M\in\mathcal{M}$ in a robust family $\mathcal{M}$ for $G$. Then the action $\rho(\Gamma)\curvearrowright M$ is free, properly discontinuous, and cocompact.  
\end{cor}

\section{Maximal $p$-submanifolds in $\mb{H}^{p,q}$}

In this section, we first quickly review the pseudo-Riemannian hyperbolic space $\mb{H}^{p,q}$ and its maximal $p$-submanifolds. Then, we prove that these manifolds form a robust family (Theorem \ref{thm:maximal is robust}). Lastly, we prove Theorem \ref{thm:main2}.

\subsection{Pseudo-Riemannian hyperbolic spaces and maximal submanifolds}

We denote by $\mb{R}^{p,q}$ the Euclidean space $\mb{R}^{p+q+1}$ equipped with the quadratic form 
\[
\langle z,z'\rangle:=x_1x_1'+\ldots+x_px_p'-y_1y_1'-\ldots-y_{q+1}y_{q+1}'.
\]

\begin{dfn}[pseudo-Riemannian hyperbolic Space]
The {\em pseudo-Riemannian hyperbolic space} $\mb{H}^{p,q}$ is 
\[
\mb{H}^{p,q}:=\{z\in\mb{R}^{p,q}\,|\,\langle z,z\rangle=-1\}
\]
equipped with the pseudo-Riemannian metric $g_{\mb{H}^{p,q}}$ of signature $(p,q)$ induced by the restriction of the quadratic form $\langle\bullet,\bullet\rangle$ to the tangent spaces $T_z\mb{H}^{p,q}=z^\perp$.    
\end{dfn}

In general, smooth submanifolds of $\mb{H}^{p,q}$ do not inherit any pseudo-Riemannian metric. Of particular importance among those for which the restriction of $g_{\mb{H}^{p,q}}$ to the tangent spaces has constant signature are the Riemannian ones (the so-called spacelike submanifolds) of maximal dimension. 

\begin{dfn}[Spacelike $p$-Submanifold]
A smooth $p$-submanifold $M\subset\mb{H}^{p,q}$ is {\em spacelike} if the restriction of $g_{\mb{H}^{p,q}}$ to $T_zM$ is positive definite for every $z\in M$. In particular, $g_{\mb{H}^{p,q}}$ induces a Riemannian metric on $M$. The spacelike $p$-submanifold is said to be {\em complete} if it is complete in the Riemannian sense.    
\end{dfn}

We can now give the definition of a maximal $p$-submanifold.

\begin{dfn}[Maximal $p$-Submanifold]
A spacelike $p$-submanifold $M\subset\mb{H}^{p,q}$ is {\em maximal} if it has vanishing mean-curvature, that is, the second fundamental form $\mb{I}_M$ of $M$ satisfies the following equation ${\rm tr}_M(\mb{I}_M)=0$.
\end{dfn}

To better understand the moduli space of spacelike or maximal $p$-submanifolds there are convenient global atlases of $\mb{H}^{p,q}$, the so-called Poincaré models, where they display a particularly nice form (for an in-depth discussion see \cite[Section 3]{SST} and \cite[Section 3]{BK}).

\begin{dfn}[Poincaré Model]
\label{dfn:poincare model}
Fix an orthogonal splitting $\mb{R}^{p,q}=E\oplus F$ with $E$ of signature $(p,0)$ and $F=E^\perp$. Denote by $\mb{D}^p\subset E$ the open unit ball and by $\mb{S}^q\subset F$ the unit sphere. There is a smooth diffeomorphism $\Psi:\mb{D}^p\times\mb{S}^q\to\mb{H}^{p,q}$ (the so-called {\em Poincaré model}) defined by
\[
\Pi(u,v):=\frac{2}{1-|u|^2}u+\frac{1+|u|^2}{1-|u|^2}v.
\]
The {\em spherical metrics} on $\mb{D}^p,\mb{S}^q$ are given respectively by 
\[
\cos(d_{\mb{D}^p}(u,u')):=\frac{4\langle u,u'\rangle+(1-|u|^2)(1-|u'|)^2}{(1+|u|^2)(1+|u'|^2)}\quad\text{\rm and}\quad\cos(d_{\mb{S}^q}(v,v')):=|\langle v,v'\rangle|.
\]
\end{dfn}

We have the following.

\begin{pro}[{see \cite[Lemma 3.11]{SST}}]
\label{pro:topology}
Let $M\subset\mb{H}^{p,q}$ be a complete spacelike $p$-submanifold. In any Poincaré model $\mb{D}^p\times\mb{S}^q\simeq\mb{H}^{p,q}$, the manifold $M$ is the graph of a smooth function $u:\mb{D}^p\to\mb{S}^q$ which is 1-Lipschitz with respect to the spherical metrics on domain and target. In particular, $M$ is diffeomorphic to $\mb{R}^p$.
\end{pro}

Proposition \ref{pro:topology} allows us to equip the space of complete spacelike $p$-submanifolds with a nice topology.

\begin{dfn}[Smooth Convergence Topology]
\label{dfn:smooth conv}
The {\em topology of smooth convergence} on the space of complete spacelike $p$-submanifolds is the topology induced by smooth convergence on compact subsets on the space of smooth functions $\mathcal{C}^\infty(\mb{D}^p,\mb{S}^q)$. This topology is independent of the choice of a Poincaré model for $\mb{H}^{p,q}$.     
\end{dfn}

\begin{lem}
\label{lem:smooth implies pointed}
Smooth convergence implies pointed $\mathcal{C}^2$-convergence.
\end{lem}

\begin{proof}
Let us work in a Poincaré model $\Pi:\mb{D}^p\times\mb{S}^q\to\mb{H}^{p,q}$. Consider a sequence $M_n$ of complete spacelike $p$-submanifolds converging to a complete smooth spacelike $p$-submanifold $M$ in the smooth convergence topology. Let $u_n,u:\mb{D}^p\to\mb{S}^q$ be the smooth functions whose graphs are $M_n,M$. The graph maps $\tau_n,\tau:x\mapsto (x,u_n(x)),(x,u(x))$ provide global charts for $\Pi^{-1}(M_n),\Pi^{-1}(M)$ in which the metric tensors are $\sigma_{M_n}=\tau_n^*\Pi^*g_{\mb{H}^{p,q}},\sigma_M=\tau^*\Pi^*g_{\mb{H}^{p,q}}$. Note that $\sigma_{M_n}\to\sigma_M$ smoothly on compact subsets. As approximating maps, we just define $f_n:=\tau_n\pi\Pi^{-1}$ where $\pi:\mb{D}^p\times\mb{S}^q\to\mb{D}^p$ is the projection to the first factor. In the global charts for $M_n,M$ given by $\tau_n\Pi^{-1},\tau\Pi^{-1}$, the approximating map $f_n$ is just the identity. It follows that $f_n$ converges smoothly on compact subsets to the inclusion $\iota:M\to\mb{H}^{p,q}$ and that $f_n^*\sigma_{M_n}\to\sigma_M$ smoothly on compact subsets.   
\end{proof}

\subsection{Maximal submanifolds are robust}
We establish the following result.

\begin{thm}
\label{thm:maximal is robust}
The family $\mathcal{M}{\rm ax}$ of complete maximal $p$-submanifolds of $\mb{H}^{p,q}$ is a robust family for ${\rm SO}(p,q+1)$.  
\end{thm}

\begin{proof}
We need to check four properties. 

{\bf Property \eqref{lab:invariance}}. This follows immediately from the definitions.    

{\bf Property \eqref{lab:compactness}}. Consider the space of pointed complete maximal $p$-submanifolds 
\[
\mathcal{PM}{\rm ax}:=\{(M\subset\mb{H}^{p,q},o\in M)\}\subset\mathcal{M}{\rm ax}\times\mb{H}^{p,q}.
\]
By \cite[Theorem 5.3]{SST}, if we equip $\mathcal{M}{\rm ax}$ with the smooth convergence topology (Definition \ref{dfn:smooth conv}), then the action of ${\rm SO}(p,q+1)$ on $\mathcal{PM}{\rm ax}$ is cocompact. Thus, it is enough to show that if $(M_n,o_n)\to(M,o)$ in the topology of smooth convergence, then it converges also in the pointed $\mathcal{C}^2$-topology in the sense of Definition \ref{dfn:pointed topology}. This is the content of Lemma \ref{lem:smooth implies pointed}.

{\bf Property \eqref{lab:avoidance}}. Consider an unbounded sequence $g_n\in{\rm SO}(p,q+1)$ such that $a_ng_n\to\phi\neq 0$ for some $a_n\to 0$. Let us start by observing the following. For every $x\in M$, the subspace ${\rm Span}\{M\}$ contains $T_xM$ which is a subspace of signature $(p,0)$. Hence, to show that $M$ cannot be contained in ${\rm Ker}(\phi)$ it is enough to prove that ${\rm Ker}(\phi)$ cannot contain a positive definite subspace of dimension $p$. As $a_ng_n\to\phi$, we have $a_ng_n^{-1}\to\phi^t$. We now observe that ${\rm Ker}(\phi)$ is orthogonal to ${\rm Im}(\phi^t)$ since for every $v\in{\rm Ker}(\phi)$ and $u\in\mb{R}^{p,q}$ we have $0=\langle \phi(v),u\rangle=\langle v,\phi^{t}(u)\rangle$. Let us prove that ${\rm Im}(\phi^t)$ is totally isotropic. For every $u,v\in\mb{R}^{p,q}$ we have
\[
\langle\phi^t(u),\phi^t(v)\rangle=\lim{a_n^2\langle g_n^{-1}(u),g_n^{-1}(v)\rangle}=\lim{a_n^2\langle u,v\rangle}=0
\]
as desired. As the image of $\phi^t=\lim{a_ng_n^{-1}}$ is a non-trivial totally isotropic subspace, its orthogonal in $\mb{R}^{p,q}$ cannot contain a subspace of signature $(p,0)$.

{\bf Properties \eqref{lab:domination1}, \eqref{lab:domination2}, and \eqref{lab:domination3}}. For every $(M,o)\in\mathcal{M}$ we set
\[
d_{(M,o)}(x):={\rm arccosh}\left(\min\left\{|\langle o,x\rangle|, 1\right\}\right).
\]

This is the so-called {\em spacelike pseudo-distance} introduced first by Glorieux and Monclair \cite{GM}. Property \eqref{lab:domination1} follows immediately from the definition.

Property \eqref{lab:domination2}. This is a basic property of complete spacelike $p$-submanifolds of $\mb{H}^{p,q}$, see for example \cite[Lemma 2.1]{MV}.

Lastly, let us prove Property \eqref{lab:domination3}. Fix a Poincaré model $\mb{H}^{p,q}\simeq\mb{D}^p\times\mb{S}^q$. Consider a sequence $M_n$. By Proposition \ref{pro:topology}, in the Poincaré model $M_n$ is the graph of a smooth function $f_n:\mb{D}^p\to\mb{S}^q$ which is 1-Lipschitz with respect to the spherical metrics on the domain and target. We choose $o_n:=f_n(o)$ where $o\in\mb{D}^p$ is any fixed basepoint. Up to subsequences, by Arzelà-Ascoli, $f_n$ converges to a 1-Lipschitz function $f_\infty:\mb{D}^p\to\mb{S}^q$. By our choices, the functions $d_{(M_n,o_n)}(x):={\rm arccosh}\left(\min\left\{|\langle o_n,x\rangle|, 1\right\}\right)$ converge to $d_\infty(x)={\rm arccosh}\left(\min\left\{|\langle o_\infty,x\rangle|, 1\right\}\right)$ where $o_\infty=f_\infty(o)$.

Together, the four properties prove the theorem.
\end{proof}

\subsection{The proof of Theorem \ref{thm:main2}}
We must prove openness and closedness.

{\bf Openness}. Our proof here does not differ from the one in \cite[Proposition 7.4]{BK} which rests on the Ehresmann-Thurston principle combined with the work of Seppi, Smith, and Toulisse \cite{SST}. We briefly sketch it. 

Observe that, by Corollary \ref{cor:cohomological dimension}, if $\Gamma$ is a torsion-free group of cohomological dimension $p$ that admits a discrete and faithful representation $\rho\in\T(\Gamma,{\rm SO}(p,q+1),\mathcal{M}{\rm ax})$, then $\Gamma$ is isomorphic to the fundamental group of an aspherical $p$-manifold. Openness comes from two facts. Let $Y$ be a closed aspherical $p$-manifold with universal cover $X\to Y$. 

By the Ehresmann-Thurston principle, if $\rho:\pi_1(Y)\to{\rm SO}(p,q+1)$ admits an equivariant spacelike embedding $X\to\mb{H}^{p,q}$ then any representation sufficiently near $\rho$ in ${\rm Hom}(\pi_1(Y),{\rm SO}(p,q+1))$ does the same.

By \cite{SST}, if $\rho:\pi_1(Y)\to{\rm SO}(p,q+1)$ admits an equivariant spacelike embedding $\iota:X\to\mb{H}^{p,q}$ then it admits one $\iota_\rho$ whose image $\iota_\rho(X)$ is a complete maximal $p$-submanifold. As the action $\rho(\pi_1(Y))\curvearrowright\iota_\rho(X)$ is by isometries and topologically conjugate to the deck group action $\pi_1(Y)\curvearrowright X$, the representation $\rho:\pi_1(Y)\to{\rm Isom}(\iota_\rho(X))$ is injective and has discrete image. By Corollary \ref{cor:discrete and faithful}, it follows that $\rho$ is discrete and faithful.

{\bf Closedness}. Let $\rho_n:\Gamma\to{\rm SO}(p,q+1)$ be a sequence of discrete and faithful representations each preserving a maximal $p$-submanifold $M_n\subset\mb{H}^{p,q}$ and converging to $\rho$ in ${\rm Hom}(\Gamma,{\rm SO}(p,q+1))$. We want to show that $\rho$ preserves a maximal $p$-submanifold. By Theorem \ref{thm:maximal is robust}, maximal $p$-submanifolds form a robust family for ${\rm SO}(p,q+1)$. Hence, by Theorem \ref{thm:main4precise} and Remark \ref{rmk:discrete and faithful}, there are $o_n\in M_n$, $g_n\in {\rm SO}(p,q+1)$, and a complete maximal pointed $p$-submanifold $(M,o)$ such that (1) $g_n(M_n,o_n)\to(M,o)$ in the pointed $\mathcal{C}^2$-topology on the space of pointed maximal $p$-submanifolds, (2) $\bar{\rho}_n:=g_n\rho_ng_n^{-1}$ converges to a discrete and faithful representation $\bar{\rho}$ preserving $M$. By Corollaries \ref{cor:discrete and faithful} and \ref{cor:cohomological dimension}, $\bar{\rho}(\Gamma)$ acts properly discontinuously, freely, and cocompcatly on $M$. 

We recall that, by the proof of Property \eqref{lab:domination} in Theorem \ref{thm:maximal is robust}, the points $o_n\in M_n$ are chosen as follows. Fix an auxiliary Poincaré model $\Pi:\mb{D}^p\times\mb{S}^q\to\mb{H}^{p,q}$ induced by an orthogonal splitting $\mb{R}^{p,q}=E\oplus F$ where $E$ has signature $(p,0)$ and $F=E^\perp$ (see Definition \ref{dfn:poincare model}). By Proposition \ref{pro:topology}, each $M_n$ is the graph of a 1-Lipschitz function $u_n:\mb{D}^p\to\mb{S}^q$ with respect to the spherical metrics on domain and target. Then $o_n$ can be taken to be $o_n:=(0,u_n(0))$ where $0\in\mb{D}^p$ is the origin. 

Let us note that, being 1-Lipschitz, $u_n$ extends continuously to the closed disk $u_n:\mb{D}^p\cup\partial\mb{D}^p\to\mb{S}^q$ and the extension is still 1-Lipschitz. By Arzelà-Ascoli, we have (up to passing to subsequences) $u_n\to u_\infty$ uniformly on $\mb{D}^p\cup\partial\mb{D}^p$ where $u_\infty:\mb{D}^p\cup\partial\mb{D}^p\to\mb{S}^q$ is again 1-Lipschitz. Denote by $M_\infty$ the graph of $u_\infty$ restricted to the open disk $\mb{D}^p$. As $M_n$ is $\rho_n(\Gamma)$-invariant and $M_n\to M_\infty$ in the Hausdorff topology, we deduce that $M_\infty$ is $\rho(\Gamma)$-invariant. We prove that the action $\rho(\Gamma)\curvearrowright M_\infty$ is free and properly discontinuous and, hence, also cocompact (as $\rho(\Gamma)$ has cohomological dimension $p$ and $M_\infty$ is contractible).

\begin{claim}
Consider $x_\infty\in M_\infty$. There is a constant $c>0$ and a finite generating set $S\subset\Gamma$ such that for every $\gamma\in\Gamma$ we have
\[
{\rm arccosh}(|\langle x_\infty,\rho(\gamma)x_\infty\rangle|)\ge c|\gamma|_S-1/c.
\]       
\end{claim}

Note that freeness and proper discontinuity of $\rho(\Gamma)\curvearrowright M_\infty$ follow.

\begin{proof}
Let $x_n\in M_n$ be a sequence of points converging to $x$. Note that $d_{M_n}(x_n,o_n)\le d_{(M_n,o_n)}(x_n,o_n)={\rm arccosh}(|\langle x_n,o_n\rangle|)\to {\rm arccosh}(|\langle x_\infty,o_\infty\rangle|)=:R$. In particular, for all $n$ large, $d_{M_n}(x_n,o_n)\le R+1$. Hence for $n$ large we have $x_n=g_n^{-1}f_n(z_n)$ where $z_n\in B_{M_\infty}(o_\infty,R+1)$ and $f_n:B_{M_\infty}(o_\infty,R+1)\to g_n(M_n)$ is the 2-bilipschitz approximating map provided by the pointed $\mathcal{C}^2$-convergence. As $f_n$ is 2-bilipschitz
\begin{align*}
d_{M_n}(x_n,\rho_n(\gamma)x_n) &=d_{M_n}(g_n^{-1}f_n(z_n),\rho_n(\gamma)g_n^{-1}f_n(z_n))\\
 &=d_{M_n}(g_n^{-1}f_n(z_n),g_n^{-1}\bar{\rho}_n(\gamma)f_n(z_n))\\
 &=d_{g_n(M_n)}(f_n(z_n),\bar{\rho}_n(\gamma)f_n(z_n))\\
 &\ge d_M(f_n(z_n),f_n^{-1}\bar{\rho}_n(\gamma)f_n(z_n))\to d_M(z,\bar{\rho}(\gamma)z).    
\end{align*}

By Milnor-Švarc, for every finite generating set $S\subset\Gamma$ there is $\kappa>0$ such that for every $\gamma\in\Gamma$ and $z\in M$ we have $d_M(z,\bar{\rho}(\gamma)z)\ge\kappa|\gamma|_S-1/\kappa$ where $|\bullet|_S$ denotes the word-length with respect to $S$. 

By Property \eqref{lab:domination2},
\[
{\rm arccosh}(|\langle x_n,\rho_n(\gamma)x_n\rangle|)\ge d_{M_n}(x_n,\rho_n(\gamma)x_n).
\]
Passing to the limit, the claim follows.
\end{proof}

Next, we prove that $M_\infty$ cannot contain the image (which we call a {\em lightlike ray}) of a lightlike geodeic $\ell:[0,\infty)\to\mb{H}^{p,q}$ of the form $\ell(t)=x+tv$ where $x\in M_\infty$ and $v\in T_x\mb{H}^{p,q}=x^\perp$ is an isotropic (or lightlike) vector, that is $\langle v,v\rangle=0$. Note that $\langle\ell(t),\ell(t')\rangle=-1$ for every $t,t'\in[0,\infty)$.

\begin{claim}
\label{claim:no lightlike}
$M_\infty$ does not contain a lightlike ray $\ell\subset M_\infty$.    
\end{claim}

\begin{proof}[Proof of the claim]
Suppose by contradiction that $M_\infty$ contains a lightlike ray $\ell$. Let $x_j\in \ell$ be a sequence of elements escaping every compact set. Let $K\subset M_\infty$ be a compact fundamental domain for the action of $\rho(\Gamma)$ on $M_\infty$ containing $x_0$. Set $A:=\max\{{\rm arccosh}(|\langle x,x'\rangle|)\,|\,x,x'\in K\}$. There is a sequence $\alpha_j$ of pairwise distinct elements such that $x_j\in\rho(\alpha_j)K$. 

Let $x_j^n\in M_n$ be a sequence of points converging to $x_j$. We have 
\begin{align*}
{\rm arccosh}(|\langle x_j^n,x_0^n\rangle|) &\ge d_{M_n}(x_j^n,x_0^n) &\text{\rm Property \eqref{lab:domination2}}\\
 &\ge d_{M_n}(\rho(\alpha_j)x_0^n,x_0^n)-d_{M_n}(\rho(\alpha_j)x_0^n,x_j^n) &\text{\rm triangle inequality}\\
 &\ge d_{M_n}(\rho(\alpha_j)x_0^n,x_0^n)-{\rm arccosh}(|\langle\rho(\alpha_j)x_0^n,x_j^n\rangle|) &\text{\rm Property \eqref{lab:domination2}}\\
 &\ge d_{M_n}(\rho(\alpha_j)x_0^n,x_0^n)-A &\rho(\alpha_j)x_0^n,x_j^n\in\rho(\alpha_j)K.
\end{align*}

Note that the left-hand side converges to ${\rm arccosh}(|\langle x_j,x_0\rangle|)=0$. 

Consider then
\[
d_{M_n}(\rho(\alpha_j)x_0^n,o_n)\le d_{M_n}(\rho(\alpha_j)x_0^n,x_0^n)+d_{M_n}(x_0^n,o_n).
\]
By the above inequality, the first summand in the right-hand side is uniformly bounded. By Property \eqref{lab:domination2}, the second term satisfies $d_{M_n}(x_0^n,o_n)\le {\rm arccosh}(|\langle x_0^n,o_n\rangle|)$ and ${\rm arccosh}(|\langle x_0^n,o^n\rangle|)\to {\rm arccosh}(|\langle x_0,o_\infty\rangle|)$. Thus, also the second term is uniformly bounded. This means that $x_0^n$ is contained in a uniform ball $B(o_n,R)$ and hence, by the pointed $\mathcal{C}^2$-convergence, $g_n(M_n,o_n)\to (M,o)$ can be written as $x_0^n=g_n^{-1}f_n(w_0^n)$ where $f_n:B_{M_\infty}(o,R)\to g_n(M_n)$ is the approximating map provided by the pointed $\mathcal{C}^2$-convergence. We have:
\begin{align*}
d_{M_n}(\rho_n(\alpha_j)x_0^n,x_0^n) &=d_{M_n}(\rho_n(\alpha_j)g_n^{-1}f_n(w_0^n),g_n^{-1}f_n(w_0^n))\\
&=d_{g_n(M_n)}(\bar{\rho}_n(\alpha_j)f_n(w_0^n),f_n(w_0^n))\\
&\ge d_M(f_n^{-1}\bar{\rho}_n(\alpha_j)f_n(w_0^n),w_0^n)/2\to d_M(\bar{\rho}(\alpha_j)w_0,w_0)/2.
\end{align*}

By Milnor-Švarc, for every finite generating set $S\subset\Gamma$ there is $\kappa>0$ such that for every $\gamma\in\Gamma$ and $w\in M$ we have $d_M(w,\bar{\rho}(\gamma)w)\ge\kappa|\gamma|_S-1/\kappa$ where $|\bullet|_S$ denotes the word-length with respect to $S$. Thus $d_{M_n}(\rho(\alpha_j)x_0^n,x_0^n)\ge|\alpha_j|_S/\kappa-1/\kappa$ diverges, contradicting the fact that it should be uniformly bounded. This concludes the proof of the claim.
\end{proof}

Lastly, we argue that the elements $g_n\in{\rm SO}(p,q+1)$ realizing $g_n(M_n,o_n)\to(M,o)$ must be bounded. From there we immediately deduce that $\rho$ and $\bar{\rho}$ are conjugate $\bar{\rho}=g\rho g^{-1}$ by some element $g\in{\rm SO}(p,q+1)$ (a suitable limit of the $g_n$'s) and, hence, that $\rho(\Gamma)$ preserves the complete maximal $p$-submanifold $g^{-1}(M)$.

Suppose by contradiction that the sequence $g_n$ is unbounded. Then there is a sequence $a_n\to 0$ such that $a_ng_n^{-1}\to\psi$ with $\psi$ non-trivial with totally isotropic image (see the proof of Property \eqref{lab:avoidance} in Theorem \ref{thm:maximal is robust}). We prove that:

\begin{claim}
If $g_n$ is unbounded, then $M_\infty$ contains a lightlike ray $\ell\subset M_\infty$. 
\end{claim}

\begin{proof}[Proof of the claim]
Recall that we fixed a Poincaré model $\Pi:\mb{D}^p\times\mb{S}^q\to\mb{H}^{p,q}$ induced by an orthogonal splitting $\mb{R}^{p,q}=E\oplus F$ where $M_\infty$ is the graph of a the restriction to $\mb{D}^p$ of a function $u_\infty:\mb{D}^p\cup\partial\mb{D}^p\to\mb{S}^q$ which is 1-Lipschitz with respect to the spherical metrics on domain and target. Recall also that $u_\infty$ is the uniform limit of the functions $u_n:\mb{D}^p\cup\partial\mb{D}^p\to\mb{S}^q$ whose graphs are the spacelike $p$-submanifolds $M_n$. By Property \eqref{lab:avoidance}, there exists $x\in M$ such that $x\not\in{\rm Ker}(\psi)$. Let $f_n:(M,o)\to (g_n(M_n),g_n(o_n))$ be the 2-bilipschitz approximating maps provided by the pointed $\mathcal{C}^2$-convergence. Consider $g_n^{-1}f_n(x)\in M_n$ and write it as 
\[
g_n^{-1}f_n(x)=\Pi(p_n,u_n(p_n))=\frac{2}{1-|p_n|^2}p_n+\frac{1+|p_n|^2}{1-|p_n|^2}u_n(p_n).
\]
Since $a_ng_n^{-1}f_n(x)\to\psi(x)=e+f$ where $e\in E,f\in F$, we have
\[
a_n\frac{2}{1-|p_n|^2}p_n\to e\quad\text{\rm and }\quad a_n\frac{1+|p_n|^2}{1-|p_n|^2}u_n(p_n)\to f.
\]

Note that, as $a_n\to 0$ and the limit is $\psi(x)\neq 0$, necessarily $|p_n|\to 1$. Note also that $a_n(1+|p_n|^2)/(1-|p_n|^2)\to |f|$. Up to subsequences we have $p_n\to p\in\mb{D}^p$ with $|p|=1$ (in particular $d_{\mb{D}^p}(0,p)=\pi/2$) and by the uniform convergence $u_n\to u_\infty$ we also get $u_n(p_n)\to u_\infty(p)$. Now, on the one hand, we have
\[
\langle a_ng_n^{-1}f_n(x),o_n\rangle=\langle a_nf_n(x),g_no_n\rangle\to 0
\]
as $f_n(x)\to x,g_n(o_n)\to o,a_n\to 0$. On the other hand,
\begin{align*}
\langle a_ng_n^{-1}f_n(x),o_n\rangle &=\left\langle \frac{2a_n}{1-|p_n|^2}p_n+a_n\frac{1+|p_n|^2}{1-|p_n|^2}u_n(p_n),u_n(0)\right\rangle\\
&\to\frac{|f|}{|u_\infty(p)|}\langle u_\infty(p),u_\infty(0)\rangle=\cos\left(d_{\mb{S}^q}(u_\infty(p),u_\infty(0))\right)
\end{align*}
Thus, $d_{\mb{S}^q}(u_\infty(p),u_\infty(0))=\pi/2=d_{\mb{D}^p}(0,p)$. Since $u_\infty$ is 1-Lipschitz with respect to the spherical metrics, and satisfies $d_{\mb{S}^q}(u_\infty(p),u_\infty(0))=d_{\mb{D}^p}(p,0)$ we deduce that $u_\infty$ is isometric on the ray $[0,p]$. It follows that the graph of $[0,p)$ is a lightlike ray (see for example \cite[Lemma 1.8]{MV2}).
\end{proof}

As the above violates Claim \ref{claim:no lightlike}, we reached a contradiction as desired. 

\section{Affine $p$-spheres in $\mb{R}^{p+1}$}

In this section, we briefly describe some convex projective geometry and its relation with affine spheres. Then, we show that hyperbolic affine spheres form a robust family (Theorem \ref{thm:affine is robust}). Lastly, we prove Theorem \ref{thm:main3}.

\subsection{Affine geometry and affine spheres}

We first recall the notion of affine normal and affine metric on a smooth hypersurface $M\subset\mb{R}^{p+1}$. Consider a vector field $\xi$ on $\mb{R}^{p+1}$ transverse to $M$, that is, for every $x\in M$ we have $\mb{R}^{p+1}=T_xM\oplus{\rm Span}\{\xi(x)\}$. The splitting allows us to decompose the standard flat connection $D$ of $\mb{R}^{p+1}$ along the tangent and transverse directions
\[
\begin{array}{c}
(D_UV)_x=(\nabla^\xi_UV)_x+\sigma^\xi_x(U,V)\xi(x)\\
(D_U\xi)_x=-S(U)_x+\tau(U)\xi(x)
\end{array}
\]
where $U,V\in T_xM$ and $\sigma^\xi_x(\bullet,\bullet)$ is a symmetric bilinear form on $T_xM$. 

\begin{dfn}[Locally Uniformly Convex]
A smooth hypersurface $M\subset\mb{R}^{p+1}$ is {\em locally uniformly convex} if for some (and hence any) choice of transverse vector field $\xi$ we have that $\sigma^\xi_x$ is positive definite at every point $x\in M$. 
\end{dfn}

Affine spheres are particular convex hypersurfaces.

\begin{dfn}[(Hyperbolic) Affine Sphere]
If $M$ is a locally uniformly convex smooth hypersurface, then there is a unique choice of transverse vector field $\xi$, the so-called {\em affine normal}, such that ${\rm dvol}_M=\iota_\xi{\rm dvol}_{{\rm Eu}}$ and $\nabla^\xi\omega=0$.

We say that $M$ is an {\em affine $p$-sphere} if its affine normals all pass through some common point $O$, the so-called {\em center of $M$}. We will consider only {\em hyperbolic affine $p$-spheres}, that is, affine $p$-spheres where the center $O$ lies on the concave side of $M$.

The metric $\sigma(\bullet,\bullet)$ associated to the affine normal is the {\em affine metric} of $M$. We say that $M$ is {\em complete} if $(M,\sigma)$ is complete as a Riemannian manifold.
\end{dfn}

Every complete hyperbolic affine sphere $M\subset\mb{R}^{p+1}$ determines an open convex cone by taking ${\rm Cone}(O,M)$. Viceversa, Cheng and Yau \cite{CY1,CY2} show that for every open convex cone $\mathcal{C}$ centered at $O$ and containing no lines there is a unique complete hyperbolic affine sphere $M$ with the same center asymptotic to it, that is $\mathcal{C}={\rm Cone}(O,M)$.

Being contained in a convex cone with no lines, every complete hyperbolic affine sphere comes equipped with a Hilbert distance.   

\begin{dfn}[Hilbert Distance]
Let $M\subset\mb{R}^{p+1}$ be a complete hyperbolic affine sphere centered at the origin. For every $x,y\in M$, the intersection ${\rm Cone}(0,M)\cap{\rm Span}\{x,y\}$ is an open sector bounded by two rays spanned by the vectors $a,b$ (on the sides of $x,y$ respectively). The {\em Hilbert distance} between $x,y$ is defined as
\[
h_M(x,y):=\log\mb{B}(a,x,y,b)
\]
where $\mb{B}$ is the cross-ratio on the projective line $\mb{P}{\rm Span}\{x,y\}$.    
\end{dfn}

Benoist and Hulin provide the following relation between the affine and Hilbert metrics which we will exploit in the proof of Theorem \ref{thm:affine is robust}.

\begin{pro}[{see \cite[Proposition 3.4]{BH}}]
\label{pro:hilbert metric}
There exists a constant $c>1$ only depending on the dimension $p$ such that for every affine $p$-sphere we have
\[
d_M(\bullet,\bullet)\le c\cdot h_M(\bullet,\bullet)
\]
where $d_M(\bullet,\bullet),h_M(\bullet,\bullet)$ are respectively the affine Riemannian distance and the Hilbert distance. 
\end{pro}

\subsection{Hyperbolic affine spheres are robust}

We prove the following.

\begin{thm}
\label{thm:affine is robust}
The family $\mathcal{A}{\rm ff}$ of complete affine $p$-spheres of $\mb{R}^{p+1}$ centered at the origin is a robust family for ${\rm SL}(p+1,\mb{R})$.  
\end{thm}

\begin{proof}
We need to check four properties. 

{\bf Property \eqref{lab:invariance}}. This follows directly from the definitions.

{\bf Property \eqref{lab:compactness}}. Consider the set of pointed complete hyperbolic affine $p$-spheres
\[
\mathcal{PA}{\rm ff}:=\{(M\subset\mb{R}^{p+1},o\in M)\}\subset\mathcal{A}{\rm ff}\times\mb{R}^{p+1}.
\]
It admits a natural embedding in the product space ${\rm Cl}(\mb{S}^p)\times\mb{S}^p$, where ${\rm Cl}(\mb{S}^p)$ is the space of closed subsets of $\mb{S}^p$, given by 
\[
(M,o)\in\mathcal{PA}{\rm ff}\to\left(\overline{\pi(M)},\pi(o)\right)\in{\rm Cl}(\mb{S}^p)\times\mb{S}^p.
\]
Here $\pi:\mb{R}^{p+1}-\{0\}\to\mb{S}^p$ the natural radial projection and $\overline{\pi(M)}$ is the closure of $\pi(M)$ in $\mb{S}^p$. We endow ${\rm Cl}(\mb{S}^p)$ with the Hausdorff topology and ${\rm Cl}(\mb{S}^p)\times\mb{S}^p$ with the product topology. By \cite{B}, if we equip $\mathcal{A}{\rm ff}$ with the topology induced by the embedding above, then the action of ${\rm SL}(p+1,\mb{R})$ on $\mathcal{PM}$ is cocompact. Thus, it is enough to show that if $(\overline{\pi(M_n)},\pi(o_n))\to(\overline{\pi(M)},\pi(o))$ in ${\rm Cl}(\mb{S}^p)\times\mb{S}^p$, then $(M_n,o_n)\to(M,o)$ also in the pointed $\mathcal{C}^2$-topology in the sense of Definition \ref{dfn:pointed topology}. This is proved in \cite[Proposition 2.8 and Corollary 3.3]{BH}.

{\bf Property \eqref{lab:avoidance}}. Consider an unbounded sequence $g_n\in{\rm SL}(p+1,\mb{R})$ such that $a_ng_n\to\phi\neq 0$ for some $a_n\to 0$. Note that we always have ${\rm Span}\{M\}=\mb{R}^{p+1}$. In fact, for every $x\in M$, the subspace ${\rm Span}\{M\}$ contains the dimension $p$ subspace $T_xM$ and the vector $x$ which is transverse to it. This implies that $M\subsetneq{\rm Ker}(\phi)$. 

{\bf Properties \eqref{lab:domination1}, \eqref{lab:domination2}, and \eqref{lab:domination3}}. For every $(M,o)\in\mathcal{PA}{\rm ff}$ we set 
\[
d_{(M,o)}(x):=\log|\phi_o(x)|
\]
where $\phi_o:\mb{R}^{p+1}\to\mb{R}$ is the linear functional that vanishes on $T_xM$ and takes value 1 on $o$. Property \eqref{lab:domination1} is immediate from the definition.

Next, we discuss Property \eqref{lab:domination2}. Consider $(M,o)\in\mathcal{PM}$ and $x\in M$. The intersection $M\cap{\rm Span}\{o,x\}$ is a curve which we can be parameterized as
\[
\gamma(t)=e^{t+\alpha(t)}u^++e^{-t+\alpha(t)}u^-
\]
where $u^+,u^-$ are the two vectors of ${\rm Span}\{o,x\}$ determining the asymptotic directions of $M\cap{\rm Span}\{o,x\}$ and $\alpha:\mb{R}\to\mb{R}$ is a smooth function. This is a unit speed geodesic parametrization of $M\cap{\rm Span}\{o,x\}$ with respect to the natural Hilbert distance on $M$. Recall that by Proposition \ref{pro:hilbert metric}, there is a constant $c>1$ only depending on the dimension $d$ such that $d_M(\bullet,\bullet)\le c\cdot h_M(\bullet,\bullet)$.

Set $\alpha_0:=\alpha(0),\dot{\alpha}_0:=\dot{\alpha}(0)$. We can assume that $o=\gamma(0)=(e^{\alpha_0},e^{\alpha_0})$ and $x=\gamma(t_x)$ with $t_x>0$. The velocity at $o$ is $\dot{\gamma}(0)=((1+\dot{\alpha}_0)e^{\dot{\alpha}_0},(-1+\alpha_0)e^{\alpha_0})$. A simple linear algebra computation shows that
\[
\phi_o(\gamma(t))=e^{\alpha(t)-\alpha_0}(\cosh(t)+\dot{\alpha}_0\sinh(t)).
\]
By \cite[Proposition 2.2]{Tho}, we have $|\dot{\alpha}|\le a:=1-1/c$ where $c>1$ is the constant such that $d_M(\bullet,\bullet)\le c\cdot h_M(\bullet,\bullet)$. Thus 
\begin{align*}
\log|\phi_o(\gamma(t))| &=(\alpha(t)-\alpha_0)+\log(\cosh(t)+\dot{\alpha}_0\sinh(t)) &\\
 &\ge(\alpha(t)-\alpha_0)+\log(\cosh(t)-a\sinh(t)) &\text{\rm as $|\dot{\alpha}|<a$}\\
 &\ge(\alpha(t)-\alpha_0)+\log(e^t(1-a)/2) &\\
 &\ge(1-a)t+\log((1-a)/2) &\text{\rm as $|\alpha(t)-\alpha_0|\le||\dot{\alpha}||t<at$.}
\end{align*}

In conclusion
\[
d_M(o,x)\le c\cdot h_M(o,x)=ct_x\le\frac{c}{1-a}\left(\log|\phi_o(\gamma(t_x))|-\log\frac{1-a}{2}\right)
\]
and Property \eqref{lab:domination2} follows.

As for Property \eqref{lab:domination3}, consider a sequence $M_n$ of affine $p$-spheres centered at the origin. Define $o_n\in M_n$ to be the point of minimal Euclidean norm. Note that this is well defined as $M_n$ does not contain 0 and is a proper convex hypersurface of $\mb{R}^{p+1}$. Note also that the tangent space of $M_n$ at $o_n$ is the Euclidean orthogonal $T_{o_n}M_n=o_n^\perp$. It follows that $\phi_{o_n}(x)=\langle o_n/|o_n|^2_{{\rm Eu}},x\rangle_{{\rm Eu}}$. The points $o_n/|o_n|_{{\rm Eu}}$ all lie on the sphere $\mb{S}^p$ and, hence, they converge to some $o_\infty\in\mb{S}^p$ up to passing to subsequences. With the same choice, for every $\gamma\in\Gamma$ we have that $\phi_{o_n}(\rho_n(\gamma)o_n)=\langle o_n/|o_n|^2_{{\rm Eu}},\rho_n(\gamma)o_n\rangle_{{\rm Eu}}\to\langle o_\infty,\rho(\gamma)o_\infty\rangle_{{\rm Eu}}$. 

Together, the four properties prove the theorem.
\end{proof}

\subsection{The proof of Theorem \ref{thm:main3}}
We must prove openness and closedness.

{\bf Openness}. Again, as in the proof of Theorem \ref{thm:main2}, openness is very general and rests on the Ehresmann-Thurston principle this time combined with the work of Cheng and Yau \cite{CY1,CY2}. We briefly sketch it.

Observe that, by Corollary \ref{cor:cohomological dimension}, if $\Gamma$ is a torsion-free group of cohomological dimension $p$ that admits a discrete and faithful representation $\rho\in\T(\Gamma,{\rm SL}(p+1,\mb{R}),\mathcal{A}{\rm ff})$, then $\Gamma$ is isomorphic to the fundamental group of an aspherical $p$-manifold. Openness comes from two facts. Let $Y$ be a closed aspherical $p$-manifold with universal cover $X\to Y$.

By the Ehresmann-Thurston principle, if $\rho:\pi_1(Y)\to{\rm SL}(p+1,\mb{R})$ admits an equivariant strictly convex embedding $\iota_\rho:X\to\mb{R}^{p+1}$ then any representation sufficiently near $\rho$ in ${\rm Hom}(\pi_1(Y),{\rm SL}(p+1))$ does the same.

By \cite{CY1,CY2}, if $\rho:\pi_1(Y)\to{\rm SO}(p+1,\mb{R})$ admits an equivariant strictly convex embedding $\iota:X\to\mb{R}^{p+1}$ then it admits one $\iota_\rho$ whose image $\iota_\rho(X)$ is a complete hyperbolic affine $p$-sphere. As the action $\rho(\pi_1(Y))\curvearrowright\iota_\rho(X)$ is by isometries and topologically conjugate to the deck group action $\pi_1(Y)\curvearrowright X$, the representation $\rho:\pi_1(Y)\to{\rm Isom}(\iota_\rho(X))$ is injective and has discrete image. By Corollary \ref{cor:discrete and faithful}, it follows that $\rho$ is discrete and faithful.

{\bf Closedness}. Let $\rho_n:\Gamma\to{\rm SL}(p+1,\mb{R})$ be a sequence of discrete and faithful representations each preserving a hyperbolic affine sphere $M_n\subset\mb{R}^{p+1}$ and converging to $\rho$ in ${\rm Hom}(\Gamma,{\rm SL}(p+1,\mb{R}))$. We want to show that $\rho$ preserves a hyperbolic affine sphere. By Theorem \ref{thm:affine is robust}, affine spheres form a robust family for ${\rm SL}(p+1,\mb{R})$. Hence, by Theorem \ref{thm:main4precise} and Remark \ref{rmk:discrete and faithful}, there are $o_n\in M_n$, $g_n\in {\rm SL}(p+1,\mb{R})$, and a complete hyperbolic pointed affine sphere $(M,o)$ such that (1) $g_n(M_n,o_n)\to(M,o)$ in the pointed $\mathcal{C}^2$-topology on the space of pointed affine spheres, (2) $\bar{\rho}_n:=g_n\rho_ng_n^{-1}$ converges to a discrete and faithful representation $\bar{\rho}$ preserving $M$. By Corollaries \ref{cor:discrete and faithful} and \ref{cor:cohomological dimension}, $\bar{\rho}(\Gamma)$ acts properly discontinuously, freely, and cocompactly on $M$. 

We argue that $g_n$ must be bounded. Suppose by contradiction that this is not the case. Then there is a sequence $a_n\to 0$ such that $a_ng_n\to\phi$ with $\phi$ non-trivial. Note that $(a_ng_n)\rho_n(\gamma)=\bar{\rho}_n(\gamma)(a_ng_n)$ for every $\gamma\in\Gamma$. Passing to the limit, we obtain $\phi\rho(\gamma)=\bar{\rho}(\gamma)\phi$ for every $\gamma\in\Gamma$. Thus $\bar{\rho}(\Gamma)$ preserves ${\rm Im}(\phi)$. The contradiction comes from the fact that if $\bar{\rho}(\Gamma)$ acts cocompactly on a hyperbolic affine sphere and $\bar{\rho}(\Gamma)$ does not contain any infinite normal nilpotent subgroup then $\bar{\rho}(\Gamma)$ acts irreducibly on $\mb{R}^{p+1}$. This property follows immediately from \cite[Proposition 4.4 (d)]{Be2} (which is a mild variation of \cite[Lemma 3.7 (b)]{Be0}). In fact, \cite[Proposition 4.4 (d)]{Be2} implies that if a group $\bar{\rho}(\Gamma)$ acts cocompactly on a hyperbolic affine sphere and does not act irreducibly on $\mb{R}^{p+1}$, then it contains a finite index subgroup with a non-trivial center of the form $\mb{Z}^\ell$ for some $\ell\ge 1$.

\bibliographystyle{amsalpha.bst}
\bibliography{bibliography}

\Addresses

\end{document}